\documentclass{amsart}
\usepackage{amsthm,amsmath,amsfonts,amssymb,enumerate,array,multirow,booktabs,stmaryrd,esint}
\usepackage{hyperref}
\newtheorem{thm}{Theorem}[section]
\newtheorem{lemma}[thm]{Lemma}
\newtheorem{corol}[thm]{Corollary}
\newtheorem{prop}[thm]{Proposition}

\theoremstyle{definition}
\newtheorem{defi}[thm]{Definition}
\theoremstyle{remark}

\newtheorem{oss}[thm]{Remark}
\newtheorem*{notation}{Notation}

\newcommand{\R}{\mathbb R}
\newcommand{\CP}{\mathbb{CP}}
\newcommand{\lie}{\mathfrak}

\DeclareMathOperator{\Rm}{Rm}
\DeclareMathOperator{\Ric}{Ric}

\DeclareMathOperator{\const}{const}

\renewcommand{\epsilon}{\varepsilon}

\renewcommand{\bar}{\overline}

\newcommand{\lieder}{\mathcal L}

\newcommand{\desharp}{\partial^\sharp}
\newcommand{\de}{\partial}
\newcommand{\debar}{{\bar \partial}}

\DeclareMathOperator{\Hess}{Hess}
\newcommand{\call}{\mathcal}

\title{On Calabi Extremal K\"ahler-Ricci solitons}
\author{Simone Calamai and David Petrecca}
\address{(S.~Calamai) Dip. di Matematica e Informatica ``U. Dini'' - Universit\`a di Firenze \endgraf Viale Morgagni 67A -  Firenze - Italy}
\email{simocala at gmail.com}
\address{(D.~Petrecca) Dipartimento di Matematica - Universit\`a di Pisa \endgraf Largo Pontecorvo 5 -  Pisa - Italy}
\email{petrecca at mail.dm.unipi.it}

\begin{document}

\begin{abstract}
In this note we give a characterization of  K\"ahler metrics which are both Calabi extremal
and K\"ahler-Ricci solitons in terms of complex Hessians and the Riemann curvature tensor.  
We apply it to prove that, under the assumption of positivity of the holomorphic sectional curvature, 
these metrics are Einstein.
\end{abstract}
\maketitle
\section*{Introduction}
Let $M^{2n}$ be a compact complex manifold.
A K\"ahler metric $g$ on $M$ is said to be \emph{K\"ahler-Einstein} if it is
Einstein as a Riemannian metric, i.e. it is proportional to its Ricci tensor or,
equivalently, if there exists $c \in \R$ such that
\begin{equation}\label{KEeq}
\rho_g = c \omega_g
\end{equation}
where $\rho_g$ (resp. $\omega_g$) denotes the Ricci form (resp. K\"ahler form)
of $g$.

There are two possible generalizations of this notion.
The first is the notion of \emph{extremal metric} introduced by  Calabi in
\cite{calabi1, calabi2} (see also \cite{gaud}) in the attempt to find a
canonical representative in a given K\"ahler class $ \Omega \in H^{1,1}(M) \cap
H^2(M, \R)$. These metrics are defined to be the critical points of the
Riemannian functional $\call M_\Omega \rightarrow \R$ defined by
\[
g \longmapsto \int_M s_g^2 \omega^n
\]
where $\call M_\Omega$ is the space of the K\"ahler metrics on $M$ in the class
$\Omega$ and $s_g$ is the scalar curvature of $g$.
He also showed that a metric is extremal if, and only if, the gradient of $s_g$
is a holomorphic vector field. 
Constant scalar curvature K\"ahler metrics (cscK), hence in particular
K\"ahler-Einstein, are examples of extremal metrics, but there are extremal
metrics of non constant scalar curvature (see again \cite{calabi1}).

Another direction to generalize the Einstein condition \eqref{KEeq} is the
following.
A K\"ahler metric $g$ is called a \emph{K\"ahler-Ricci soliton} (KRS) if there
exist $c \in \R$ and a holomorphic vector field $X$ such that
\begin{equation}
\rho_g +c \omega_g = \lieder_X \omega_g. \nonumber
\end{equation}
These metrics give rise to special solutions of the K\"ahler-Ricci flow (see e.g.
\cite{chow}), namely they evolve under biholomorphisms.
It is known that on a compact manifold, if $c \geq 0$ then $g$ is Einstein (see
e.g. again \cite{chow}), so in the compact K\"ahler case one only considers the
so-called \emph{shrinking} K\"ahler-Ricci solitons ($c<0$) whose equation, after
a scaling, can be written as
\begin{equation}\label{KRSeq}
\rho_g - \omega_g = \lieder_X \omega_g.
\end{equation}
The Hodge decomposition for the dual of $X$ allows us to introduce a holomorphy
potential  with respect to $g$, i.e. a complex-valued function $\theta_X$ such that $\iota_X
\omega_g = i \debar \theta_X$. By means of this function we can infer that the
K\"ahler form $\omega_g$ belongs to $2 \pi c_1(M)$, making $M$ a Fano manifold.

The first examples of non-Einstein compact K\"ahler-Ricci solitons go back to
the constructions of  Koiso \cite{koiso} and independently Cao \cite{cao} of
K\"ahler metrics on certain $\CP^1$-bundles over $\CP^n$.
Koiso himself remarks that this  K\"ahler-Ricci soliton metric is not Calabi
extremal and proves that if it were, it would be Einstein.

There is a class of manifolds for which there are existence results for both kinds of metrics, namely
toric manifolds (see e.g. \cite{abreu}). For extremal metrics we mention for
instance the existence result of  Donaldson \cite{donaldson_ext} for toric
surfaces. For the KRS we refer to the existence result, in all dimensions, of
Wang and  Zhu \cite{wz}. Finally, the existence of 
a K\"ahler-Einstein metric on a compact Fano
manifold is equivalent to the notion of \emph{K-stability}
stated by  Chen,  Donaldson and  Sun in \cite{songKE} and subsequent papers.

It is then natural to ask what happens when a metric generalizes a K\"ahler-Einstein metric in both these ways. 
In Theorem \ref{thmprinc} we prove, under the assumption of positivity of the holomorphic sectional curvature, 
that an extremal KRS is in fact K\"ahler-Einstein.

It is not too restrictive to assume positive holomorphic sectional curvature provided it does not exceed a certain numerical bound. Indeed it has been proved by Futaki and Sano \cite{futakisano} that if the diameter of a Ricci soliton is $< \frac{10}{13} \pi$ then the soliton is trivial.
On the other hand, Tsukamoto \cite{tsukamoto} proved that if a K\"ahler metric $g$ has  holomorphic sectional curvature $> \epsilon$, then the diameter of the manifold is bounded from above by $\frac{\pi}{\sqrt \epsilon}$.

From these results we can infer that if the holomorphic sectional curvature is greater than $ (\frac{13}{10})^2$ then the KRS must be Einstein, so our result is non trivial when the holomorphic sectional curvature does not exceed this number. The authors do not know whether there are any connections between positive holomorphic sectional curvature and the extremality condition or whether the extremality gives conditions on the diameter .

The paper is organized as follows. We start recalling some notation and conventions of K\"ahler geometry. 
This level of detail seems necessary in order to avoid confusion among different conventions. 
We then go on proving, for a K\"ahler-Ricci soliton, the characterization of being extremal in terms 
of the length of the complex Hessian of its  potential function and in terms of certain contractions of the Riemann curvature tensor.
We then use this to prove our main result in Theorem \ref{thmprinc} and we also give a condition about the isometry group of a non-Einstein extremal KRS. We finally make a remark about the replicability of the argument in the Sasakian setting.

\subsection*{Acknowledgements} The authors are grateful to Fabio Podest\`a for suggesting the problem, for his constant advice and support and for his help for a better presentation of this paper. They are also grateful to Xiuxiong Chen for his support and to Weiyong He and Song Sun for their interest and feedback.

\section{Definitions and preliminary results}
\begin{notation}
 Let $(M, g, J)$ be a smooth, compact, without boundary K\"ahler manifold of
 real dimension $2n$, with its Riemannian metric $g$ and compatible integrable complex structure $J$.
 The corresponding K\"ahler form is $\omega (\cdot , \, \cdot )= g(J\cdot , \, \cdot)$. 
 We also denote as $\Ric$ the Riemannian Ricci tensor corresponding to the Riemannian metric $g$;
 and its Ricci form as $\rho (\cdot , \, \cdot) = \Ric (J \cdot, \cdot)$.
We label  $s$ the Riemannian scalar curvature  of the metric $g$. 
We let $\delta$ be adjoint of the  exterior differential $d$ with respect to $g$ 
and $\Delta_d = \delta d + d \delta$ be the $d$-Laplacian
acting on differential forms.  

We let $\sharp$ and $\flat$ denote the musical isomorphisms between fields and $1$-forms.
For a $1$-form $\alpha$ we denote as $|\alpha|^2 =( \alpha ,\, \alpha) = (\alpha^\sharp ,\, \alpha^\sharp)$, 
and as well as $|Z|^2 = (Z,\, Z) = (Z^\flat , \, Z^\flat)$ the metric pairing 
on by means of the Riemannian metric $g$.
Similarly, if a real $(1,1)$-form $\beta$ and a $2$-tensor $B$
correspond each other via $\beta (\cdot , \, \cdot ) = B(J\cdot,\, \cdot)$,
then we have for the metric pairings $|B|^2 =(B,\, B)=2|\beta|^2 = 2(\beta, \, \beta)$.
For example, $|\Ric|^2 =(\Ric,\, \Ric)=2|\rho|^2 = 2(\rho, \, \rho)$.
Notice also that for any smooth real valued function on $M$ there holds $(\omega , dd^c u) = -\Delta_d u$.

Given any tensor $T$ and any vector field $V$ on a smooth manifold, we label as $\lieder_V T$
the Lie derivative of $T$ along $V$.

For any smooth, real valued function $u$ on $M$, we label as $\nabla u$ the Riemannian gradient of $u$;
namely, $\nabla u = (d u )^\sharp$. We also denote its $(1,0)$-part as $\desharp f = \frac 1 2 ( \nabla f - i J \nabla f)$. We let $\Hess u = \frac{1}{2}\lieder_{\nabla u}g$ be the real Hessian of $u$.

We also label as $\lie h(M)$ the algebra of (complex) holomorphic vector fields of $M$.

\end{notation}

The first definition is very classical and tracks  back Hamilton \cite{hamilton}.

\begin{defi}
 \label{KRS} Let $(M, g, J)$ and $\Ric$ as in \emph{Notation}; 
 let $f$ be a smooth, real valued function on $M$.
 We say that $(g,\, f)$ is a K\"ahler-Ricci soliton when the following equation is satisfied
 \begin{align}
  \label{KRSequation}
  \Ric - g = \frac{1}{2}\lieder_{\nabla f} g.
 \end{align}
\end{defi}
\begin{oss}\label{oss:holomorphypotentialofKRShasholomorphicgradient}
It is a general fact that $\nabla f$ is real holomorphic, although this is often stated in the definition. Indeed, equation \eqref{KRSequation} implies that $\nabla_i \nabla_j f = 0$ for all 
 $i,\, j\in \{1, \cdots , \, n\}$.
\end{oss}

The next definition is due to Calabi \cite{calabi1}.

\begin{defi}\label{defi:extremalmetric}
 Let $(M, g, J)$ and $s$ be as in \emph{Notation}.  
 We say that the metric $g$ is \emph{Calabi extremal}, or simply extremal, when the Riemannian gradient of $s$
 is holomorphic, i.e. if  $\debar \de^\sharp s = 0$. 
\end{defi}
In this paper we consider metrics which satisfy both these definitions.
\begin{defi}
 \label{defi:extremalKRS}
 Let $(M,\, g\, J)$ as in \emph{Notation}, and let $(g, f)$ be a K\"ahler-Ricci soliton as in
 Definition \ref{KRS}. Moreover, let $g$ be Calabi extremal as in Definition \ref{defi:extremalmetric}.
 Then, we call $(g, f)$ an extremal K\"ahler-Ricci soliton.
\end{defi}

\begin{oss}
 We chose to label the pairs $(g,\, f)$ in Definition \ref{defi:extremalKRS} as extremal K\"ahler-Ricci
 solitons although a similar name was given by Guan in \cite{guan} to different objects.
\end{oss}

Not all complex valued smooth  functions $v$ on $M$ give rise to holomorphic vector fields. The
ones which do are solutions of the equation $\debar \desharp v = 0$, they lie in
the kernel of the fourth order differential operator $L_g = (\debar \desharp)^*
\debar \desharp$ (see \cite{FutakiBook}).

The presence of an extremal metric gives 
information about the algebra of holomorphic vector fields $\lie h(M)$. Namely
the following theorems hold.

\begin{thm}[{\cite{calabi2, FutakiBook}}] \label{thm:calabi}
Let $g$ be an extremal K\"ahler metric on $M$ with scalar curvature $s$. Then
the Lie algebra $\lie h(M)$ has a semidirect sum decomposition
\begin{equation} \label{deco_extr}
\lie h(M) = \lie a (M) \oplus \lie h'(M),
\end{equation}
where $\lie a(M)$ is the complex Lie subalgebra of $\lie h(M)$ consisting of all
parallel holomorphic vector fields of $M$, and $\lie h'(M)$ is an ideal of $\lie
h(M)$ consisting of the image under $\desharp$ of the kernel of $L_g$.

Moreover $\lie h'(M)$ has a decomposition
\[
\lie h'(M) = \bigoplus_{\lambda \geq 0} \lie h_\lambda(M)
\]
where $[\desharp s , Y] = \lambda  Y$ for any $Y \in \lie h(M)$. Furthermore
the centralizer $\lie h_0(M)$ of $\desharp s$ is the complexification of the Lie
algebra consisting of Killing vector fields of $M$.
\end{thm}

In the case of a K\"ahler-Ricci soliton a similar theorem holds. 
\begin{thm}[{\cite{tz_uniq}}] \label{thm:TZ}
 If $g$ is a K\"ahler-Ricci soliton with $(1,0)$-vector field $X$. Then $\lie
h(M)$ admits the decomposition
\begin{equation} \label{deco_KRS}
\lie h(M) = \lie k_0(M) \oplus \bigoplus_{\lambda > 0} \lie k_\lambda(M)\, ,
\end{equation}
where $\lie k_\lambda(M) = \{ Y \in \lie h(M): [X, Y] = \lambda Y \}$. Moreover
the centralizer $\lie k_0(M)$ of $X$ splits as $\lie k_0' \oplus \lie k_0''$
where $\lie k_0'$ is the $\desharp$-image of real functions and $\lie k_0''$ is
the $\desharp$-image of purely imaginary functions.
\end{thm}

The following  result is due to Lichnerowicz (see \cite[Proposition 2.140]{besse}).

\begin{prop}
 \label{prop:lichnerowicz characterization}
On a compact K\"ahler manifold a (real) vector field $X$ is holomorphic if and only if 
\begin{align}
 \label{equa:lichnerowicz characterization}
 \Delta_d X^\flat - 2 \Ric (X,\, \cdot ) = 0\, .
\end{align}

\end{prop}

\section{Statements and proofs}

The function $f$ in Definition \ref{KRS} has, 
by means of Remark \ref{oss:holomorphypotentialofKRShasholomorphicgradient}, 
holomorphic gradient so it satisfies, applying $\delta$ on both sides of 
\eqref{equa:lichnerowicz characterization} (cf. \cite[(1.17.5)]{gaud}) 
  
\begin{equation} \label{Gaudf}
\frac 1 2 \Delta_d^2 f + (dd^c f, \omega + \frac 1 2 dd^c f) + \frac 1 2 (ds, df) =0 \; .
\end{equation}
 Tracing the KRS equation \eqref{KRSequation} we get $2n-s =  \Delta_d f$, 
 we substitute it into \eqref{Gaudf} to get 
 \begin{equation*}
\frac 1 2 \Delta_d (2n-s) - \Delta_d f + \frac 1 2 |dd^c f|^2 + \frac 1 2 (ds, df) =0 \; . 
 \end{equation*}
So
\begin{equation}\label{equa:tointegrate}
- \Delta_d s + 2(s-2n)+ |dd^c f|^2 + (ds, df) =0 \; . 
\end{equation}
Differentiating we get 
\begin{equation} \label{dM}
- \Delta_d ds + 2ds + d  |dd^c f|^2 +d (ds, df) =0\; .
\end{equation}

The last term in \eqref{dM} can be substituted with the following two lemmas.

\begin{lemma}\label{lemma:commutation}
On an extremal K\"ahler-Ricci soliton $(g,\, f)$ the holomorphic fields $\nabla f$ and $\nabla s$ commute.
\end{lemma}
\begin{proof}
Since $(g, f)$ is an extremal K\"ahler-Ricci soliton, then both $\de^\sharp s$ and $\de^\sharp f$
are holomorphic vector fields, i.e. $\de^\sharp s ,\, \de^\sharp f \in \lie{h}(M)$. 
Also, by means of Theorem \ref{thm:calabi}, $\lie{h}(M)$ splits as 
$\lie{h}(M) = \lie{a}(M) \oplus \bigoplus_{\lambda \geq 0} \lie h_\lambda(M) $.
The summand $\lie h_0 (M)$,  the centralizer of $\de^\sharp s $ in $\lie{h}(M)$,
contains $\de^\sharp f$.
Hence we have 
\begin{align*}
 0&=[\de^\sharp s , \, \de^\sharp f] =\frac{1}{4}\left( [\nabla s,\, \nabla f] - [J\nabla s,\, J\nabla f]
 -i [\nabla s,\, J\nabla f] - i [J\nabla s,\, \nabla f]\right) \\
 &=\frac{1}{2}\left( [\nabla s,\, \nabla f] 
 -iJ [\nabla s,\, \nabla f] \right)\, ,
\end{align*}
and we take its real part to conclude.
\end{proof}

\begin{lemma} \label{lemmadg}
Whenever two real functions $u,v$ satisfy $[\nabla u \, , \nabla v] = 0$, then there holds
\begin{align*}
d(g(\nabla u, \nabla v)) = (\nabla_{\nabla u} \nabla v + \nabla_{\nabla v} \nabla u)^\flat 
= 2(\nabla_{\nabla u} \nabla v)^\flat  \; . 
\end{align*}
\end{lemma}
\begin{proof}
For any vector field $Y$ we have
\begin{align*}
Y \cdot g(\nabla v, \nabla u)	&= g(\nabla_Y \nabla v, \nabla u) + g(\nabla v, \nabla_Y \nabla u) \\
						&= g( \nabla_{\nabla u} \nabla v + \nabla_{\nabla v} \nabla u, Y) \\
						&= g(2 \nabla_{\nabla u} \nabla v, Y) \; .
\end{align*}
This completes the proof of the lemma.
\end{proof}
For a (K\"ahler-)Ricci soliton there are some quantities that are constant, see e.g. \cite{chow}. One of them is, in our notation,
\begin{equation} \label{eqchow}
s + |\nabla f|^2 + 2f = \const.
\end{equation}
From this together with Lemma \ref{lemmadg} it is easy to infer the following.
\begin{lemma}Let $g$ be a KRS with real holomorphic field $X$ and let $Z = X^{1,0}$.
Then $g$ is extremal if, and only if, $\nabla_X X$ is real holomorphic (or $\nabla_Z Z$ is holomorphic or $\nabla_{\bar Z} \bar Z$ is antiholomorphic).
\end{lemma}

At this point it is worth noticing the following.
\begin{prop}For an extremal KRS $g$ with field $X$ and scalar curvature $s$, if $X = c \nabla s$ then $g$ is Einstein.
\end{prop}
\begin{proof} We first notice that $c$ has to be constant on $M$. Indeed if it were a function on $M$ it would be holomorphic since the two fields are.
Consider the function $p \mapsto |X_p|^2$ and a local maximum $q \in M$. At $q$ we would have
\[
 g_q ( \nabla_X X |_q, X_q) =0.
\]
Under the proportionality assumption \eqref{eqchow} becomes 
\begin{equation} \label{chowc}
(c+2) X + 2 \nabla_X X = 0.
\end{equation}
At $q$ we would have then $\frac{c+2} 2 g_q (X, X) = 0$ which implies $X = 0$ if $c \neq -2$.

If $c = -2$ we have from \eqref{chowc} and Lemma \ref{lemmadg} that $\nabla_X X = \nabla |X|^2 =0$ implying $X=0$ as well.
\end{proof}

 By means of the decomposition theorems \ref{thm:calabi} and \ref{thm:TZ}, the fields $JX$ and $J \nabla s$ belong to the center of the isometry algebra and are linearly independent for a non-Einstein extremal KRS. This gives us the following corollary.
 
\begin{corol} If $g$ is a non-Einstein extremal KRS, then the center of the isometry group of $g$ has dimension at least $2$. 
\end{corol}

We now present a characterization of extremal K\"ahler-Ricci solitons.

\begin{prop}\label{prop:characterizationextremalkrs}
Let $(M^{2n}, g, \omega, f)$ be a compact K\"ahler-Ricci soliton with Riemannian scalar curvature $s$.  Let $X= \nabla f$.
Then the following are equivalent.
\begin{enumerate}
\item \label{ddccost} the function $|dd^cf|^2$ is constant and $[\nabla f, \nabla s]=0$;
\item \label{extr} $g$ is extremal;
\item \label{curvaturetwospots} $\Rm(\cdot, \bar{\desharp f}) \bar{\desharp f} = 0$; 
\item \label{comm&forma11} The tensor 
\[
T_X:= \Rm(\cdot, X) X
\]
commutes with $J$ and $\alpha: (A, B) \mapsto \Rm(A, JX, X, B)$ is a $(1,1)$-form.
\end{enumerate}
\end{prop}
\begin{proof}
Let us first prove the equivalence between (\ref{ddccost}) and (\ref{extr}).
By means of Proposition \ref{prop:lichnerowicz characterization}, 
the condition on $g$ being extremal is equivalent to require the Riemannian scalar curvature $s$ 
to satisfy the tensorial Lichnerowicz equation
(see \cite[Proposition 2.140]{besse})
\begin{equation} \label{Lich}
\Delta_d ds - 2 \Ric(\nabla s, \cdot)= 0\; .
\end{equation}
Let us assume $g$ to be extremal. Equation \eqref{Lich}, together with the K\"ahler-Ricci soliton assumption
$\Ric = g+ \Hess f$, reads
 \begin{align}
0&=\Delta_d ds - 2 ds - 2 \Hess_f(\nabla s, \cdot)	 \nonumber\\
&=\Delta_d ds - 2 ds - 2 g(\nabla_{\nabla s} \nabla f, \cdot) \; .  \label{equa:lichtensons}
\end{align}
By means of Lemma \ref{lemmadg}, formula \eqref{equa:lichtensons} differs from \eqref{dM} 
by the term $d | dd^c f|^2$ which has to be zero.

Conversely, assuming $[\nabla s ,\, \nabla f]=0$, then Lemma \ref{lemmadg} holds. 
Also, in \eqref{dM} the term $d|dd^c f|^2$ vanishes and then \eqref{dM} is simply \eqref{Lich},
which says, by means of Proposition \ref{prop:lichnerowicz characterization}, that $s$ has holomorphic gradient.
This completes the equivalence between (\ref{ddccost}) and (\ref{extr}).

Let $g$ be extremal, then by means of the previous Lemma, the field $\nabla_Z Z$ where $Z = \desharp f$ is holomorphic.
Then compute for any $(1,0)$-field $A$,
\begin{align*}
\Rm(A, \bar Z) \bar Z		&= \nabla_A \nabla_{\bar Z} \bar Z - \nabla_{\bar Z} \nabla_A \bar Z - \nabla_{[\bar Z, A]} \bar Z \\
					&= 0
\end{align*}
by using the fact that $\bar Z$ and $\nabla_{\bar Z} \bar Z$ are antiholomorphic (hence killed by $\nabla_A$) that kills the first two terms and that $[\bar Z, A]$ is $(1,0)$ that kills the last.

Conversely, the generic form of the Riemann tensor for $A$ of type $(1,0)$ is
\[
\Rm(A, \bar Z) \bar Z = \nabla_A \nabla_{\bar Z} \bar Z.
\]
If this is zero, it means that the field $\nabla_{\bar Z} \bar Z$ is killed by $\nabla_A$ for any $A$ of type $(1,0)$ implying that it is antiholomorphic. Indeed, for any $W$ we have
\[
0=\nabla_{Y-iJY}(W+ iJW) = \nabla_Y W + J \nabla_{JY} W + i(J \nabla_Y W - \nabla_{JY} W)
\]
that is, $W$ satisfies $\nabla_{JY} W = J \nabla_Y W$ for all $Y$, that is $W$ is real holomorphic.

Hence we conclude $g$ is extremal by means of the previous Lemma.

Let us now assume \eqref{curvaturetwospots}. We notice that its real formulation is given by the system
\begin{equation} \label{rmreale}
\begin{cases}
\Rm(A, X) JX = - \Rm(A, JX) X \\
\Rm(A, X) X = \Rm(JA, X) JX.
\end{cases}
\end{equation}

and the second equation means exactly that $[T_X, J]=0$.
We have now, using \eqref{rmreale} for the second equality,
\begin{align*}
\alpha(B, A)	&= \Rm(B, JX, X, A) \\
			&= -\Rm(B, X, JX, A) \\
			&= -\Rm(JX, A, B, X) \\
			&= - \alpha(A, B).
\end{align*}

To prove that $\alpha$ is $(1,1)$ is equivalent to prove that it is $J$-invariant. This follows again from \eqref{rmreale} since
\begin{align*}
\alpha(JA, JB)	&= \Rm(JA, JX, X, JB) \\
			&= - \Rm(A, X, JX, B) \\
			&= \Rm(A, JX, X, B) \\
			&= \alpha(A, B).
\end{align*}

Conversely let $[T_X, J] = 0$ and let $\alpha$ be $J$-invariant. These assumption are exactly \eqref{rmreale}.
\end{proof}

We can use this to prove our main result.
\begin{thm} \label{thmprinc}
Any extremal K\"ahler-Ricci soliton with positive holomorphic sectional curvature is Einstein.
\end{thm}
\begin{proof}
Let $f$ be the soliton function. Assume it is not a constant and  $Z$ be the normalized  $\desharp f$. 

By assumption we have, in the direction $Z$, that the holomorphic sectional curvature is
\[
K(Z):= \Rm(Z, \bar Z, Z, \bar Z) > 0.
\]
By means of the previous proposition we are lead to the contradiction $K(Z) = 0$ as the above Riemann tensor vanishes.
\end{proof}

\begin{oss} \label{rmkBerger}
There is no loss of generality to assume the positivity of the holomorphic 
sectional curvature instead of just requiring it to have a sign. 
Indeed, by a theorem of Berger \cite[Lemme (7.4) pag. 50]{bergerholsec} prescribing the sign of the holomorphic sectional curvature gives the same sign to the scalar curvature, which in case of Ricci solitons is always positive by means of general results (see e.g. again \cite{chow}).
\end{oss}

The argument exposed in this paper can be replicated verbatim to prove the following result about Sasaki manifolds. We refer to \cite{FOW, canonical} for the notions of Sasaki-Ricci solitons, to and Sasaki-extremal metrics and for transverse curvature. Recall that, for a Sasaki manifold, being transversally K\"ahler-Einstein is equivalent to being $\eta$-Sasaki-Einstein.

\begin{thm}
Any extremal Sasaki-Ricci soliton with positive transverse holomorphic sectional curvature is $\eta$-Sasaki-Einstein.
\end{thm}

Indeed there are Sasakian analogues of Theorem \ref{deco_extr} done by Boyer and Galicki and an extension of Theorem \ref{thm:TZ} done by the second author in \cite{petrecca}. Moreover, the Lichnerowicz equations hold as well for the transverse quantities, see again \cite{canonical}.

\bibliography{../biblio}{}
\bibliographystyle{amsplain}
\end{document}